\documentclass{amsart}
\usepackage{amsmath,amssymb,amscd,amsthm,amsfonts}
\usepackage{graphicx}
\usepackage{hyperref}
\usepackage{pstricks}
\usepackage[section]{placeins}
\usepackage{latexsym}
\usepackage{epsfig}
\usepackage{ulem} 
\usepackage{enumerate}
\usepackage{listings}
\usepackage{lstautogobble}
\newtheorem{theorem}{Theorem}
\newtheorem{lemma}{Lemma}
\newtheorem{proposition}{Proposition}
\newtheorem{corollary}{Corollary}

\newcommand{\ba}{\boldsymbol{a}}
\newcommand{\bb}{\boldsymbol{b}}
\newcommand{\bc}{\boldsymbol{c}}
\newcommand{\bd}{\boldsymbol{d}}

\newcommand{\p}{\boldsymbol{p}}
\newcommand{\x}{\boldsymbol{x}}

\newcommand{\q}{\boldsymbol{q}}

\newcommand{\cc}{\boldsymbol{c}}
\newcommand{\R}{\mathbb{R}}

\newcommand{\Z}{\mathbb{Z}}

\DeclareMathOperator{\conv}{conv}

\DeclareMathOperator{\Tv}{Tv}
\DeclareMathOperator{\He}{H}

\renewcommand{\Re}{\mathbb{R}}

\def\P{\ensuremath{\mathcal{P}}}

\usepackage{changepage}

\newcommand{\remark}[1]
{
\begin{adjustwidth}{0.1in}{0.0in}
\textbf{Remark: }
#1 
\end{adjustwidth}
}

\newcommand{\mystate}[1]{
\vspace{0.1in}
\centerline{\begin{minipage}[t]{0.9 \textwidth}
#1 \end{minipage}}
}
\begin{document}

\title{Tverberg theorems over discrete sets of points}

\author{J. A. De Loera}
\address{Department of Mathematics, University of California, Davis}
\email{deloera@math.ucdavis.edu, tahogan@math.ucdavis.edu}
\author{T. A. Hogan}
\author{F. Meunier}
\address{CERMICS, Ecole Nationale des Ponts et Chauss\'ees, Univ. Paris-Est, France}
\email{frederic.meunier@enpc.fr}
\author{N. H. Mustafa}
\address{Laboratoire d'Informatique Gaspard-Monge, Univ. Paris-Est, ESIEE Paris, Marne-la-Vallée, France}
\email{mustafan@esiee.fr}


\begin{abstract} This paper discusses Tverberg-type theorems with coordinate constraints (i.e., versions of 
these theorems where all points lie within a subset $S \subset \mathbb{R}^d$ and the intersection of convex hulls is required to have a non-empty intersection
with $S$). We determine the $m$-Tverberg number, when $m \geq 3$, of any discrete subset $S$ of $\mathbb{R}^2$ (a generalization of an unpublished result of J.-P. Doignon).
We also present improvements on the upper bounds for the Tverberg numbers of $\mathbb{Z}^3$ and $\mathbb{Z}^j \times \mathbb{R}^k$
and an integer version of the well-known positive-fraction selection lemma of J. Pach.
\end{abstract}

\maketitle
\section*{Introduction}

Consider $n$ points in $\R^d$ and a positive integer $m\geq 2$. If $n\geq (m-1)(d+1)+1$, the points can always be partitioned into 
$m$ subsets whose convex hulls contain a common point. This is the celebrated theorem of Tverberg \cite{Tverberg:1966tb}, which 
has been the topic of many generalizations and variations since it was first proved in 1966 \cite{barany+soberonsurvey,jesus+xavier+frederic+nabil}. 
In this paper we focus on new versions of Tverberg-type theorems where some of the coordinates of the points are restricted to discrete subsets of a Euclidean space.  The associated  discrete Tverberg numbers are much harder to compute than their classical real-version counterparts
(see for instance the complexity discussion of \cite{Onn:1991wz}). 

We begin our work remembering the following unpublished Tverberg-type result of Doignon.
Consider $n$ points with coordinates in $\Z^2$ and a positive integer $m\geq 3$. If $n\geq 4m-3$, then 
the points can be partitioned into $m$ subsets whose convex hulls contain a common point in $\Z^2$.
According to Eckhoff~\cite{Eckhoff:2000jw} this result was stated by Doignon in a conference.
  
A partition of points where the intersection of the convex hulls contains at least one lattice point is called an {\em integer $m$-Tverberg partition} and 
such a common point is an {\em integer Tverberg point} for that partition.  
Regarding the case $m=2$,  the integer $2$-Tverberg partitions are  called {\em integer Radon partitions}. Any configuration of at least six 
points in $\mathbb{Z}^2$ admits an integer Radon partition. This was proved by Doignon in his PhD thesis~\cite{doignonthesis} and later discovered independently by Onn~\cite{Onn:1991wz}. All these values for $\mathbb{Z}^2$ are optimal as shown by following examples. The $5$-point configuration $\left\{(0,0),(0,1),(2,0)(1,2),(3,2)\right\}$, exhibited by Onn in the cited paper, has no Radon partition. To address the optimality when $m\geq 3$, consider the set $\{(i,i),(i,-i+1)\colon i = -m+2, -m+3 , \ldots, m-2, m-1\}.$ (According to Eckhoff~\cite{Eckhoff:2000jw}, this set was proposed by Doignon during the aforementioned conference.) 
This set has $4m-4$ points and a moment of reflection might convince the reader that it has no integer $m$-Tverberg partition.




More generally, one can define the {\em Tverberg number} $\Tv(S,m)$ for any subset $S$ of $\R^d$ and an integer $m\geq 2$ as
the smallest positive integer $n$ with the following property: Any multiset of $n$ points in $S$ admits a partition into $m$ subsets $A_1,A_2,\ldots, A_m$ with 
\[\left(\bigcap_{i=1}^m\conv(A_i)\right)\cap S\neq\varnothing.\] 
  (Here, by ``partition of a multiset'', we mean 
that each element of a multiset $A$ is contained in a number of submultisets $A_1, \dots, A_m$ so that the sum of its multiplicities in the $A_i$ is equal to its multiplicity in $A$.)  If no such number exists, we say that $\Tv(S,m) = \infty$. Note that Doignon's theorem, together with the discussion that follows, allows us to say 
\[\Tv(\Z^2,m)=\left\{\begin{array}{ll} 6 & \mbox{if $m=2$}, \\ 4m-3 & \mbox{otherwise.}\end{array}\right.\]

\subsection*{Our contributions}

Our first main result generalizes Doignon's theorem. We determine the exact $m$-Tverberg number (when $m$ is at least three) for any discrete subset $S$ of $\R^2$, as considered in \cite{integertverberg2017}. Before stating this result we recall the \emph{Helly number $\He(S)$} of a discrete subset $S$ of $\R^d$ as the smallest positive integer with the following property: Suppose $\mathcal{F}$ is a finite family of convex sets in $\mathbb{R}^d$, and that $\cap \mathcal{G}$ intersects $S$ in at least one point for every subfamily $\mathcal{G}$ of $\mathcal{F}$ having at most $\He(S)$ members. Then $\cap \mathcal{F}$ intersects $S$ in at least one point. If no such integer exists, 
we say that $\He(S) = \infty$. 

Then we have the following theorem. (The theorem is stated for $S$ with finite Helly number, as any $S \subset \R^d$ with $\He(S) = \infty$ has $\Tv(S,m) = \infty$ for all $m \geq 2$~\cite{hellylevi}.)

\begin{theorem}~\label{Stverberg}

Suppose $S$ is a discrete subset of $\R^2$ with $\He(S) < \infty$. 
If $m \geq 3$, then $$\Tv(S,m) = \He(S)(m-1) + 1.$$
Regarding the case $m = 2$, if $\He(S) \leq 3$, then $$\Tv(S,2) = \He(S) + 1,$$ and if $\He(S) \geq 4$, then $$\He(S) + 1 \leq \Tv(S,2) \leq \He(S) + 2,$$

and both values are possible.  

\end{theorem}

In particular we present a proof of Doignon's theorem, the special case of Theorem~\ref{Stverberg} where $S = \mathbb{Z}^2$.
%
%

\remark{Theorem~\ref{Stverberg} shows that $S$-Tverberg numbers of planar sets are very closely related to $S$-Helly numbers (see \cite{Ave2013, Averkovetal-tightbounds} and the references there). However, for the case $\He(S) = 4$, the bounds on $\Tv(S,2)$ given above cannot be improved. For example, $S' = \{(0,0),(0,1),(1,0),(1,1)\}$ and $\mathbb{Z}^2$ both have Helly number  four, but $\Tv(\mathbb{Z}^2,2) = 6$, while the pigeonhole principle implies that $\Tv(S',2) = 5$.}
\medskip

Our second main result improves the upper bound on the integer Tverberg numbers for the three-dimensional case $S = \Z^3$.
\begin{theorem} \label{3dZ} The following inequality holds for all $m\geq2$:
\[\Tv(\Z^3, m) \leq 24m- 31.\]
\end{theorem}

Our third main result is an inequality that will be used to derive improved bounds on $S$-Tverberg numbers when $S$ is a product of a Euclidean space with some subset $S'$ of a Euclidean space. 

\begin{theorem} \label{mixedint}
Let $S' \subset \R^j$ be a subset of a Euclidean space. Then for all positive integers $k$ and all $m\geq 2$, we have
 \[\Tv(S' \times \R^k, m) \leq \Tv(S' , \Tv(\R^k,m)).\]
\end{theorem}

For example, choosing $S$ of the form $\Z^j \times \R^k$ leads to the ``mixed integer'' case. Then Theorem~\ref{mixedint} implies that for all positive integers $j, k$ and all $m\geq 2$, we have
\[\Tv(\Z^j \times \R^k, m) \leq \Tv(\Z^j , \Tv(\R^k,m)).\]
Moreover, we will use Theorem~\ref{mixedint} to obtain the following bound: 
\begin{equation}~\label{mixedbound}
2^j(m-1)(k+1) + 1 \leq \Tv(\Z^j \times \R^k,m) \leq j2^j(m-1)(k+1) + 1.\end{equation}

%
%

Our fourth main result is a generalization of \emph{Pach's positive-fraction selection lemma}~\cite{P98} (see \cite{karasevetal} for related bounds). 
Here is Pach's result: Given an integer $d$, there exists a constant $c_d$ such that for any set $P$ of $n$ points in $\Re^d$, there exists 
a point $\q \in \Re^d$, and $(d+1)$ disjoint subsets of $P$, say $P_1, \ldots, P_{d+1}$, such that $|P_i| \geq c_{d} \cdot n$ for all $i$ and
the simplex defined by every transversal of $P_1, \ldots, P_{d+1}$ contains $\q$. (By ``transversal", we mean a set containing exactly one element from 
each $P_i$.)


Unfortunately the point $\q$ need not be an integer point; furthermore, the proof uses the so-called ``second selection lemma" that currently does not exist for integer points (see Pach~\cite{P98} and Matou\v{s}ek~\cite[Chapter 9]{M02}). In Section~\ref{IPFSL}, we strengthen the above theorem, such that, 
as a consequence, the theorem now extends to the integer case---indeed, to any scenario where one has points of high half-space depth in the following sense:

Given a finite set $P$ of points in $\Re^d$ and a point $\q \in \Re^d$, we say that $\q$
is \emph{of half-space depth $t$ with respect to $P$} if any half-space containing $\q$
contains at least $t$ points of $P$ (when the context is clear,
we will simply say that $\q$ is of \emph{depth $t$}). Then here is our theorem.

\begin{theorem}
For any integer $d \geq 1$ and real number $\alpha \in (0, 1]$, there exists a constant
$c_{d, \alpha}$ such that the following holds.
For any set $P$ of $n$ points in $\mathbb{R}^d$ and any point $\q \in \Re^d$ of half-space depth at least $\alpha \cdot n$,
there exist $(d+1)$ disjoint subsets of $P$, say $P_1, \ldots, P_{d+1}$, such 
that
  \begin{itemize}
  \item $|P_i| \geq c_{d, \alpha} \cdot n$ for   $i = 1, \ldots, (d+1)$, and
  \item every   simplex defined by a transversal of $P_1, \ldots, P_{d+1}$ contains $\q$.
  \end{itemize}
\label{thm:depthcovering}
\end{theorem}

\remark{Our proof yields a constant $c_{d, \alpha}$ whose value is exponential in the dimension $d$.}

Note that the existence of integer points of high half-space depth (Lemma~\ref{centerpt}) together with Theorem~\ref{thm:depthcovering} implies the following integer version of the positive-fraction selection lemma.

\begin{corollary}
Let $P$ be a set of $n \geq (d+1)$ points in $\mathbb{Z}^d$. Then there exists a
point $\q \in \mathbb{Z}^d$, and $(d+1)$ disjoint subsets of $P$, say $P_1, \ldots, P_{d+1}$,
such that $|P_i| \geq c_{d, 2^{-d}} \cdot n$ for all $i = 1, \ldots, (d+1)$, and
the simplex defined by every  transversal of $P_1, \ldots, P_{d+1}$ contains $\q$.
\end{corollary}

\remark{In particular, this implies that $\q$ belongs to many \emph{distinct} Tverberg partitions---at
least $\left(\lceil c_{d, 2^{-d}} \cdot n \rceil !\right)^d$  \emph{distinct} Tverberg partitions,
with each such Tverberg partition containing $\left\lceil c_{d, 2^{-d}} \cdot n \right\rceil$ sets.}

\medskip

%
%
%
%
\subsection*{Related  Results and Organization of the paper}
The problem of computing the Tverberg number for $\Z^d$ with $d\geq 3$ seems to be challenging. 
It has been identified as an interesting problem since the 1970's~\cite{problemsinconvexity} and yet the following inequalities 
are almost all that is known about this problem: for the general case, De Loera et al.~\cite{integertverberg2017} proved
\begin{equation}\label{DL}
2^d(m-1)+1\leq\Tv(\Z^d,m)\leq d2^d(m-1)+1, \quad \mbox{for $d\geq 1$ and $m\geq 2$.}
\end{equation} Two special cases get better bounds:
\begin{equation}\label{special}
\Tv(\Z^3,2)\leq 17 \qquad\mbox{and}\qquad 5\cdot 2^{d-2}+1\leq\Tv(\Z^d,2)\quad \mbox{for $d\geq 1$.}
\end{equation}
The left-hand side inequality is due to Bezdek and Blokhuis~\cite{Bezdek2003} and the right-hand side was 
proved by Doignon in his PhD thesis (and rediscovered by Onn).

Previously established bounds for the ``mixed integer'' case include  the bounds for the Radon number ($2$-Tverberg number) found by Averkov and Weismantel~\cite{AW2012}. \[ 2^j(k+1) + 1 \leq \Tv(\Z^j \times \R^k, 2) \leq (j+k)2^j(k+1)- j -k + 2.\]
Later, De Loera et al.~\cite{integertverberg2017} gave the following general bound for all mixed integer Tverberg numbers: 
\[\label{known-mixed} \Tv(\Z^j \times \R^k, m) \leq (j+k)2^j(m-1)(k+1)+ 1.\]
Note that \eqref{mixedbound} above is a simultaneous improvement of both of these.

Previous bounds and related work on more general $S$-Tverberg numbers can also be found in \cite{integertverberg2017}, including the following bound for any discrete $S\subset \R^d$:
$$\Tv(S,m) \leq \He(S)(m-1)d + 1.$$

The following lemma about integer points of high half-space depth is used throughout the paper. See \cite{basu+oertel} for a proof and related results.

\begin{lemma} \label{centerpt} Consider a multiset $A$ of points in $\Z^d$. If $|A|\geq 2^d (m-1)+1$ (counting multiplicities), then there is a point $\q \in \mathbb{Z}^d$ of half-space depth $m$ in $A$.
\end{lemma}

The paper is organized as follows.
In Section~\ref{TS}, we prove Theorem ~\ref{Stverberg} using a somewhat similar strategy to Birch's proof of the planar case of the original Tverberg theorem~\cite{Birch:1959ii}. In Section~\ref{TZ3}, we prove Theorem~\ref{3dZ} using techniques reminiscent of those in \cite{integertverberg2017}. In Section~\ref{CONS}, we  prove Theorem \ref{mixedint} (adapting an approach by Mulzer and Werner~\cite[Lemma 2.3]{MW13}) and collect some consequences of the main theorems presented above, including (\ref{mixedbound}). Finally, in Section \ref{IPFSL}, we prove Theorem~\ref{thm:depthcovering} by proving a new lemma and adapting the methods of Pach in \cite{P98}.
 
\section{Tverberg Numbers over Discrete Subsets of $\mathbb{R}^2$: Proof of Theorem~\ref{Stverberg}}~\label{TS}
We start with the proof of the special case $S = \Z^2$ because it nicely illustrates the techniques of the more general proof of Theorem~\ref{Stverberg}.
\subsection{Proof of the special case $S = \Z^2$} \label{TZ2}
The theorem will follow easily from the following two lemmas, the first covering the case $m\geq 3$ and the second the case $m=2$.



\begin{lemma}\label{Tplane1}
Consider a multiset $A$ of points in $\Z^2$ with $|A| \geq 4m-3$ and $m \geq 3$. If $\p \notin A$ is a point of depth $m$, then there is an $m$-Tverberg partition of $A$ with $\p$ as Tverberg point. 
\end{lemma}

\begin{lemma}\label{Tplane2}
Consider a multiset $A$ of points in $\Z^2$ with $|A| \geq 6$. If $\p \notin A$ is a point of depth two, then there is a Radon partition of $A$ with $\p$ as Tverberg point. 
\end{lemma}

\begin{proof}[Proof of special case $S = \Z^2$.]
Consider a multiset $A$ of at least $4m-3$ points in $\mathbb{Z}^2$. By Lemma~\ref{centerpt}, $A$ has an integer point $\p$ of depth $m$. If $\p$ is an element of $A$ with multiplicity $\mu$, then take the singletons $\{\p\}$ as $\mu$ of the sets in the Tverberg partition. Then $\p$ is a point of depth $m-\mu$ of the remaining $4m-\mu-3$ points. If $\mu \geq m$, we are done, and if $\mu = m-1$, the point $\p$ is in the convex hull of the remaining points and we take them to be the last set in the desired partition. If $\mu \leq m-3$, according to Lemma~\ref{Tplane1}, there is an $(m-\mu)$-Tverberg partition of the remaining points with $\p$ as Tverberg point. There is thus an $m$-Tverberg partition of $A$ with $\p$ as Tverberg point. The case $\mu = m - 2$ is treated similarly with the help of Lemma~\ref{Tplane2} in place of Lemma~\ref{Tplane1}.
\end{proof}

\begin{proof}[Proof of Lemma~\ref{Tplane1}]
Since $\p$ is not in $A$, up to a radial projection, we can assume that the points of $A$ are arranged in a circle around $\p$. Define $q$ and $r$ to be respectively the quotient and the remainder of the Euclidean division of $|A|$ by $m$. Define moreover $e$ to be $\lceil \frac{r}{q} \rceil$.

Suppose first that $\p$ is a point of depth $m + e$. In such a case, we arbitrarily select  a first point in $A$, and label clockwise the points with elements in $[m]$ according to the following pattern:
\[1, 2, \ldots, m, 1, 2, \ldots, e , 1 , 2 , \ldots, m , 1, 2, \ldots, e , 1, 2 , \ldots m, 1, 2 , \ldots k,\]
where $k=|A|-qm-(q-1)e$. Note that we have $k\leq e$.
Each half-plane delimited by a line passing through $\p$ contains at least $m+e$ consecutive points in this pattern and thus has at least one point with each of the $m$ different labels. Partitioning the points so that each subset consists of all points with a fixed label, we therefore obtain an $m$-Tverberg partition with $\p$ as Tverberg point.

Suppose now that $\p$ is not a point of depth $m + e$. There is thus a closed half-plane $H_+$, delimited by a line passing through $\p$, with $|H_+ \cap A| < m + e$. The complementary closed half-plane to $H_+$, which we denote by $H_-$, is such that $|H_- \cap A| >  4m-3 - (m+e)$. Define $\ell$ to be $|H_- \cap A|$. Since $e\leq \frac m 3$, we have $\ell\geq 2m$. Denote the points in $H_- \cap A$ by $\x_1,\ldots,\x_{\ell}$, where the indices are increasing when we move clockwise. We label $\x_i$ with $r + i$ from $\x_1$ to $\x_{m-r}$, and then label $\x_{m-r+j}$ with $j$ from $\x_{m-r +1}$ to $\x_m$. We then continue labeling the points of $A$, still moving clockwise, using labels $1, 2, \ldots , m , \ldots, 1, 2 , \ldots m, 1, 2, \ldots r.$ See Figure~\ref{fig:labeling} for an illustration of the labeling scheme.

\begin{figure}
  \centering
  \includegraphics[width=10cm ]{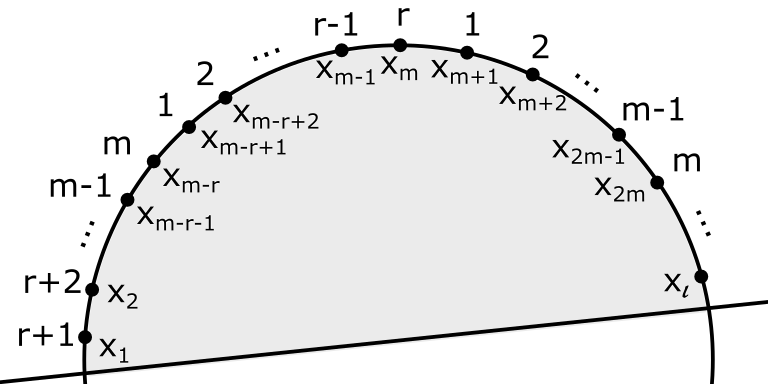}
  \caption{Labeling of the points in the half-plane $H_-$.}
  \label{fig:labeling}
\end{figure}
The labeling pattern is such that any sequence of $m$ consecutive points either has all $m$ labels, or contains the two consecutive points $\x_m$ and $\x_{m+1}$. Let us prove that any closed half-plane $H$ delimited by a line passing through $\p$ contains at least one point with each label. Once this is proved, the conclusion will be immediate by taking as subsets of points those with same labels, as above.

If such an $H$ does not simultaneously contain $\x_m$ and $\x_{m+1}$, then $H$ contains at least one point with each label. Consider thus a closed half-plane $H$ delimited by a line passing through $\p$ and containing $\x_m$ and $\x_{m+1}$. Note that according to Farkas' lemma (\cite{Sch03} Theorem 5.3) , $\x_{m+1}$ cannot be separated from $\x_1$ and $\x_{\ell}$ by a line passing through $\p$, since they are all in $H_-$. This means that either $H$ contains $\x_1, \x_2, \ldots, \x_{m + 1}$, or $H$ contains $\x_{m+1}, \x_{m+2}, \ldots, \x_{\ell}$. In any case, $H$ contains a point with each label.
\end{proof}

\begin{proof}[Proof of Lemma \ref{Tplane2}]
As before, we assume that the points in $A$ are arranged on a circle centered at $\p$. If $|A|$ is even, it clearly suffices to label the points in order, alternating between 1 and 2. We may therefore assume that $|A|$ is odd, and thus $|A| \geq 7$. If $\p$ is a point of depth three, it suffices to label the points alternating labels between 1 and 2, except with two consecutive points labeled 1.
If $|A|$ is odd but $\p$ is not a point of depth three, then $|A| \geq 7$ and there is a half-plane $H_+$ containing $\p$ with $|H_+ \cap A| = 2$. The complementary half-plane $H_-$ has $|H_- \cap A| \geq 5$ and we follow a similar strategy as in the second half of Lemma~\ref{Tplane1}. Namely, we denote the points in $H_- \cap A$ by $\x_1,\ldots,\x_{\ell}$, where the indices are increasing when we move clockwise. Then we label $\x_1$ with 2, $\x_2$ with 1, $\x_3$ with 1, and $\x_4$ with 2. We continue this pattern for $\alpha \geq 5$, labeling $\x_\alpha$ with 1 if $\alpha$ is odd, and $\x_\alpha$ with 2 if $\alpha$ is even. For the remaining points in $A$ we continue labeling clockwise, alternating between the labels 1 and 2. 

The labeling pattern is such that any sequence of $2$ consecutive points either has both labels, or contains the two consecutive points $\x_2$ and $\x_3$. As in Lemma~\ref{Tplane1} it suffices to show that any closed half-plane $H$ delimited by a line passing through $\p$ contains at least one point with each label. 

If such an $H$ does not simultaneously contain $\x_2$ and $\x_3$, then $H$ contains at least one point with each label. Consider thus a closed half-plane $H$ delimited by a line passing through $\p$ and containing $\x_2$ and $\x_3$. Note that according to Farkas' lemma, $\x_{3}$ cannot be separated from $\x_1$ and $\x_4$ by a line passing through $\p$, since they are all in $H_-$. This means that either $H$ contains $\x_1, \x_2, \x_3$, or $H$ contains $\x_{3}$ and $\x_4$. In any case, $H$ contains a point with each label.
\end{proof}

\subsection{Proof of the general case}~\label{TSA}

The proof of the general case is split into three lemmas addressing the lower bound, the upper bound for $\He(S) \geq 4$, 
and the upper bound for $\He(S) \leq 3$, respectively.

\begin{lemma}~\label{Stverberglowerbound} For any discrete set $S\subset \R^2$ with finite Helly number $\He(S)$, we have $\Tv(S,m) > \He(S)(m-1)$.
\end{lemma}

\begin{proof} It suffices to exhibit a subset $R \subseteq S$, of cardinality $|R| = \He(S)(m-1)$, with the property that no point in $S$ is of half-space depth $m$ with respect to $R$. By Lemma 2.6 in \cite{amentaetal2017helly}, there exists a set $R'$ of $\He(S)$ points in $S$ in convex position with the property that $\conv(R')\cap S = R'$. Let $R$ be the multiset given by taking each point in $R'$ with multiplicity $(m-1)$, so $|R| = \He(S)(m-1)$. No points of $S-R$ are in $\conv(R)$. Since $R'$ was taken to be in convex position, for any point in $R$, there exists a line such that one side of that line has at most $m-1$ points in $R$. Thus $S$ cannot contain a point of half-space depth $m$ with respect to $R$.
\end{proof}

\begin{lemma}~\label{helly4} For any discrete set $S\subset \R^2$ with finite Helly number $\He(S) \geq 4$, we have $\Tv(S,m) \leq \He(S)(m-1) + 1$ whenever $m \geq 3$. For the case $m = 2$, we have $\Tv(S,2) \leq \He(S) + 2$. 
\end{lemma}

\begin{proof}
The proof of Lemma~\ref{helly4} is the same as the proof of Theorem 1. In particular, we can use Lemmas 2 and 3 as they are stated, except that we use the following result (Theorem 2 in \cite{basu+oertel} with $\mu$ being the uniform probability measure on $A$) in place of Lemma~\ref{centerpt}. For any discrete discrete subset $S$ of a Euclidean space with finite Helly number $\He(S)$, and any set $A \subseteq S$ with $|A| \geq \He(S) (m-1) + 1$, there  exists a point $\p \in S$ that is of half-space depth $m$ with respect to $A$.
\end{proof}

\begin{lemma} For a discrete set $S\subset \R^2$ with finite Helly number $\He(S) \leq 3$, we have $\Tv(S,m) \leq \He(S)(m-1) + 1$.
\end{lemma}
\begin{proof}
The case $\He(S) = 1$ implies that $S$ consists of a single point, so the result trivially follows. If $\He(S) = 2$, it must be that all points in $S$ are collinear (as any set containing a non-degenerate triangle has Helly number at least 3), and thus we can take median of any set with at least $2(m-1) + 1$ points in $S$ as the desired $m$-Tverberg point. Thus for the remainder of the proof we assume that $\He(S) = 3$.

Given any set $A$ of $\He(S)(m-1) = 3m-2$ points in $S$, there exists an $m$-Tverberg partition, say $\mathcal{P}$ by the classical Tverberg theorem. We denote by $K_1, \dots, K_m$ the $m$ convex hulls of the subsets in $\mathcal{P}$. As $\bigcap_{1 \leq i \leq m} K_i$ is a nonempty polygon, say $Q$, (possibly just a point or line segment) we pick an arbitrary vertex $\q$ of $Q$.

It suffices to show that $\q \in S$. We can assume that $\q$ is not a vertex of any $K_i$, since otherwise $\q \in A \subseteq S$.

Since $\q$ is a vertex of $Q$, it must be contained in a one dimensional face $F_1$ of at least one $K_i$. Since $\q$ is not a vertex of any $K_i$, in fact $\q$ is in the relative interior of $F_1$. For $\q$ to be a vertex of $Q$, it must also be in another one dimensional face, say $F_2$, of some other $K_i$, such that $F_1$ is not parallel to $F_2$. Moreover, $\q$ must be in the relative interior of $F_2$, and we also have $F_1 \cap F_2 = \{\q\}$.

Denote by $\{\ba,\bb\}$ and $\{\bc,\bd\}$ the vertices of $F_1$ and $F_2$ respectively. We have that $\ba,\bb,\bc,\bd \in S$ are the vertices of a convex quadrilateral with diagonals intersecting at $\q$, by the assumption that $F_1$ and $F_2$ are non parallel. Out of the four triangles $\conv(\{\ba,\bb,\bc\}), \conv(\{\ba,\bb,\bd\}), \conv(\{\ba,\bc,\bd\}), \conv(\{\bb,\bc,\bd\})$, any three have at least one vertex in common, and therefore intersect in $S$. Since $\He(S) = 3$, the four triangles therefore all intersect in $S$. This intersection point is $\q$, the point where the diagonals of the quadrilateral intersect. 
\end{proof}

\section{Tverberg Numbers over $\Z^3$: Proof of Theorem~\ref{3dZ}}\label{TZ3}

We state the following lemma without proof; it is a consequence, upon close inspection of the argument, of the proof of the main theorem in the already mentioned paper by Bezdek and Blokhuis~\cite{Bezdek2003}.

\begin{lemma}\label{BB}
Consider a multiset $A$ of at least $17$ points in $\mathbb{R}^3$ and a point $\p$ of depth $3$ in $A$. There is a bipartition of $A$ into two subsets whose convex hulls contain $\p$.
\end{lemma}

Next, we prove the following.

\begin{lemma}\label{BBplus}
Consider a multiset $A$ of points in $\R^3$ with $|A| \geq 24m-31$ and $m\geq 2$. If $\p \notin A$ is a point of depth $3m-3$, then there is an $m$-Tverberg partition of $A$ with $\p$ as Tverberg point.
\end{lemma}
\begin{proof}
Since $\p$ is not an element of $A$, we assume without loss of generality that the points of $A$ are located on a sphere centered at $\p$, as in the proof of Theorem~\ref{Stverberg}. 

We claim that there exist pairwise disjoint subsets $X_1, X_2, \ldots, X_{m-2}$ of $A$, each having $\p$ in its convex hull and each being of cardinality at most $4$. (Here ``pairwise disjoint'' means that each element of $A$ is present in a number of $X_i$'s that does not exceed its multiplicity in $A$.) We proceed by contradiction. Suppose that we can find at most $s < m-2$ such subsets $X_i$'s. Then, by Carath\'eodory's theorem, $\p$ is not in the convex hull of the remaining points in $A$. Therefore there is a half-space $H_+$ delimited by a plane containing $\p$ such that 
$H_+ \cap A \subseteq \bigcup_{i=1}^s X_i$. On the other hand, since each $X_i$ contains $\p$ in its convex hull (and we can assume the $X_i$ are minimal with respect to containing $\p$), we have $|H_+ \cap X_i |\leq 3$ for all $i\in[s]$. Therefore $|H_+ \cap A| \leq \left|H_+ \cap  \left(\bigcup_{i=1}^s X_i\right)\right| \leq 3s < 3(m-2)$, which is a contradiction since $\p$ is a point of depth $3m-3$ in $A$. There are thus $m-2$ disjoint subsets $X_1, X_2, \ldots, X_{m-2}$ as claimed.

Let $X$ denote $\bigcup_{i=1}^{m-2}X_i$. Consider an arbitrary half-space $H_+$ delimited by a plane containing $\p$. Since $|H_+ \cap X_i| \leq 3$ for all $i$, we have $|H_+ \cap X| \leq 3(m-2)$. Furthermore $|H_+ \cap A| \geq 3m-3$, so $|H_+ \cap (A \setminus X) | \geq 3$. Since $H_+$ is arbitrary, $\p$ is a point of depth $3$ of $A\setminus X$. Also, $|A\setminus X| \geq |A| - 4(m-2) \geq 20m- 23 \geq 17$, so Lemma~\ref{BB} implies that $A\setminus X$ can be partitioned into two sets whose convex hulls contain $\p$. With the subsets $X_i$, we have therefore an $m$-Tverberg partition of $A$, with $\p$ as Tverberg point.
\end{proof}
From these two lemmas we can now finish the proof of Theorem~\ref{3dZ}.
\begin{proof}[Proof of Theorem~\ref{3dZ}]
Consider a multiset $A$ of $24m-31$ points in $\mathbb{Z}^3$. The case $m=2$ is the already mentioned result by Bezdek and Blokhuis. Assume that $m \geq 3$. Applying Lemma~\ref{centerpt}, $A$ has an integer point $\p$ of depth $3m-3$. If $\p$ is an element of $A$ with multiplicity $\mu$, then take the singletons $\{\p\}$ as $\mu$ of the sets in the Tverberg partition. 

If $\mu\geq m$, we are done. If $\mu=m-1$, the point $\p$ is still in the convex hull of points in $A$, and thus we are done. And if $\mu \leq m-2$, the point $\p$ is still a point of depth $3m-\mu-3\geq 3(m-\mu)-3$ of the remaining $24m-\mu-31 \geq 24(m-\mu)-31$ points. Thus, we may apply Lemma~\ref{BBplus} to get an $(m-\mu)$-Tverberg partition of the remaining points, with $\p$ as Tverberg point, and conclude the result.
\end{proof}
\section{Tverberg Numbers over $Q \times \R^k$: Proof of Theorem \ref{mixedint}}\label{CONS}
In this section, we prove Theorem~\ref{mixedint}. We adapt an approach by Mulzer and Werner~\cite[Lemma 2.3]{MW13} and show how the results of our paper 
can be combined to improve known bounds and to determine new exact values for the Tverberg number in the mixed integer case, as well as better bounds for certain $S$-Tverberg numbers. 
\begin{proof}[Proof of Theorem \ref{mixedint}] Let $t = \Tv(\R^k,m) = (m-1)(k+1) + 1$. Choose a multiset $A$ in $S' \times \R^k$ with $|A| \geq\Tv(S',t)$. It suffices to prove that $A$ can be partitioned into $m$ subsets whose convex hulls contain a common point in $S' \times \R^k$.

Let $A'$ be the projection of $A$ onto $S'$. Since $|A'| \geq\Tv(S' , t)$, there is a partition of $A'$ into $t$ submultisets $Q'_1,\ldots,Q'_t$ whose convex hulls contain a common point $\q$ in $S'$. The $Q'_i$ are the projections onto $S'$ of $t$ disjoint subsets $Q_i$ forming a partition of $A$.  For each $i\in[t]$, we can find a point $\q_i\in\conv(Q_i)$ projecting onto $\q$.

The $t$ points $\q_1,\ldots,\q_t$ belong to $\{\q\}\times\R^k$. As $t=\Tv(\R^k,m)$, there exists a partition of $[t]$ into $I_1,\ldots,I_m$ and a point $\p\in\{\q\}\times\R^k$ such that $\p\in\conv\left(\bigcup_{i\in I_{\ell}}\q_i\right)$ for all $\ell\in[m]$. For each $\ell\in [m]$, define $A_{\ell}$ to be $\bigcup_{i\in I_{\ell}}Q_i$. We have, for each $\ell\in[m]$
\[\p\in\conv\left(\bigcup_{i\in I_{\ell}}\q_i\right)\subseteq\conv\left(\bigcup_{i\in I_{\ell}}\conv(Q_i)\right)=\conv(A_{\ell})\]
 and the $A_{\ell}$ form the  desired partition.
\end{proof}

Here are the new bounds and exact values we get:

\begin{enumerate}
\item $\Tv(\mathbb{Z} \times \mathbb{R}^k,m) = 2(m-1)(k+1) +1$.
\item $\Tv(\mathbb{Z}^2 \times \mathbb{R}^k,m) = 4(m-1)(k+1)+1$.
\item $\Tv(\mathbb{Z}^3 \times \mathbb{R}^k,m) \leq 24(m-1)(k+1) -7$.
\item\label{last} $2^j(m-1)(k+1) + 1 \leq \Tv(\Z^j \times \R^k,m) \leq j2^j(m-1)(k+1) + 1$.
\item If $S' \subset \R^2$ with finite Helly number $\He(S')$, then $$\Tv(S' \times \mathbb{R}^k,m) \leq \He(S')(m-1)(k+1)+1.$$
\end{enumerate}

The lower bound in \eqref{last} is obtained by repeated applications of Proposition~\ref{stack} below. The upper bounds follow from Theorem~\ref{mixedint}, combined with the fact that $\Tv(\Z,m) = 2m-1$, Theorem~\ref{Stverberg} for $S = \Z^2$, Theorem~\ref{3dZ}, the upper bound in Equation~\eqref{DL}, and Theorem~\ref{Stverberg} respectively.
\begin{proposition}~\label{stack}
Let $j$ and $k$ be two non-negative integers. Then we have 
 \[\Tv(\Z^{j+1} \times \R^k,m) > 2\Tv(\Z^j \times \R^k, m)-2.\]
\end{proposition}
We prove Proposition~\ref{stack} by following the idea of the proof of Proposition 2.1 in \cite{Onn:1991wz}.
\begin{proof}[Proof of Proposition~\ref{stack}]
Assume toward a contradiction that
$$\Tv(\Z^{j+1} \times \R^k,m) \leq 2\Tv(\Z^j \times \R^k, m)-2.$$ Choose
$A$ to be a set of $\Tv(\Z^j \times \R^k,m) -1$ points in $\Z^j \times \R^k$ with no $m$-Tverberg partition. 
Let $A_i = \{(\boldsymbol{a},i) \colon \boldsymbol{a} \in A,i = \{0,1\}\}.$ Since $A_0 \cup A_1 \subset \Z^{j+1} \times \R^k$ has cardinality $2\Tv(\Z^j \times \R^k, m)-2$, there exists an $m$-Tverberg partition $Y_1, Y_2, \dots, Y_m$ of $A_0 \cup A_1$ with $\p \in \bigcap_{i \in [m]} \conv(Y_i)$. Furthermore $\p$ is in $\Z^{j+1} \times \R^k$. That implies either $\p \in \conv(A_0)$ or $\p \in \conv(A_1)$. In either case $A_0$ or $A_1$ has an $m$-Tverberg partition, a contradiction with our choice of $A$. 
\end{proof}

\section{A Generalized Fraction Selection Lemma: Proof of Theorem~\ref{thm:depthcovering}}\label{IPFSL}

Our proof relies on 
the   simplicial partition theorem of Matou\v{s}ek, used in a  similar
manner as in~\cite{MR17}, which states  the following.
\begin{theorem}[\cite{M92}; see also~\cite{Cha00}]~\label{simplicialpartition}
 Given an integer $d \geq 1$ and a parameter $r$, there exists a constant $c_d \geq 1$ such that
for any set $P$ of $n$ points in $\Re^d$, 
there exists an integer $s$
and a partition $\left\{P_1, \ldots, P_{s}\right\}$ of $P$ such that
\begin{itemize}
\item for each $i = 1, \ldots, s$, $\frac{n}{r} \leq |P_i|  \leq   \frac{2n}{r} $, and
\item any hyperplane intersects the convex hull of less than $c_d \cdot r^{1-\frac{1}{ d }}$ sets of the partition. 
\end{itemize}
The constant $c_d$ is independent of $P$ and depends only on $d$.
\end{theorem}

We now prove the following key lemma.

\begin{lemma}
For any integer $d \geq 1$, there exists a constant
$c_{d}$ such that the following holds.
For any set $P$ of $n$ points in $\mathbb{R}^d$ and a real number $\alpha \in (0, 1]$,
there exists a partition $\mathcal{P} = \left\{ P_1, \ldots, P_{r} \right\}$,   
$r = \left\lceil \left(\frac{4c_d}{\alpha}\right)^d \right\rceil$, of
$P$ such that
\begin{itemize}
\item $\frac{n}{2 r} \leq |P_i|  \leq \frac{2n}{r} $ for each $i = 1, \ldots, r$, and 
\item the convex hull of any transversal $Q$ of $\mathcal{P}$ contains all
points in $\mathbb{R}^d$ of half-space depth at least $\alpha \cdot n$.
\end{itemize}
\label{lemma:depthpartitioning}
\end{lemma}
\begin{proof}
Apply the simplicial partition theorem (Theorem~\ref{simplicialpartition}) to $P$ with 
$r = \left\lceil \left(\frac{4c_d}{\alpha}\right)^d \right\rceil$, and
let the resulting partition be $\left\{P'_1, \ldots, P'_s\right\}$.  
Note that as $\frac{n}{r} \leq |P'_i|  \leq   \frac{2n}{r} $ for each $i = 1, \ldots, s$,
we have $\frac{r}{2} \leq s \leq r$. Now partition arbitrarily each of 
$r-s$ biggest sets in $\left\{P'_1, \ldots, P'_s\right\}$ into two equal parts,
and let the resulting partition be $\left\{P_1, \ldots, P_r\right\}$.  
Clearly each set of this partition has size in the interval $\left[ \frac{n}{2r}, \frac{2n}{r}\right]$.
This proves the first part.
Note also that each hyperplane intersects the convex hull of at most twice as many sets, i.e., 
less than $2 c_d \cdot r^{1-\frac{1}{ d }}$ sets of the partition $\left\{P_1, \ldots, P_r\right\}$.

To see the second part, let
$\cc$ be any  point of half-space depth at least $\alpha \cdot n$, and
$Q$ any transversal of $\P$. For contradiction, assume
that $\cc \notin \mathrm{conv}\left(Q\right)$. Then there exists
a hyperplane $H$ containing $\cc$ in one of its two open half-spaces, say $H^-$, 
and containing $\mathrm{conv}\left(Q\right)$ in the half-space $H^+$.
We will show that then there exists an index $i \in \left\{1, \ldots, r\right\}$
such that $P_i \subseteq H^-$. But then $P_i \cap Q = \varnothing$, a contradiction to
the fact that $Q$ is a transversal of $\P$.

It remains to show the existence of a set $P_i \in \P$ such that $P_i \subseteq H^-$. 
Towards this, we bound $\left| P \cap H^- \right|$. Each point  of $P$
lying in $H^-$  belongs to a set $P' \in \P$ such that either
\begin{itemize}
\item $P'\subseteq H^-$, in which case we are done, or
\item $P' \nsubseteq H^-$. As $H^-$ contains at least one point of $P'$, we must have $\mathrm{conv}\left(P' \right) \cap H \neq \varnothing$. As argued earlier, there are less than
$2 c_d \cdot r^{1-\frac{1}{d}}$ such sets.
\end{itemize}
Thus we have
\begin{align}
\left|P \cap H^- \right| <  2 c_d  \cdot r^{1-\frac{1}{d}} \cdot   \frac{2n}{r}   
= \frac{4 c_d\cdot n  }{\left\lceil  \left( \frac{4c_d }{\alpha}\right)^d\right\rceil^{\frac{1}{d}}}  \leq \alpha \cdot n.
\label{eq:upperbound}
\end{align}
On the other hand, as $\cc$ has half-space depth at least $\alpha \cdot n$ and $\cc \in H^-$, 
we have $\left| P \cap H^- \right| \geq \alpha n$,
a contradiction to inequality~(\ref{eq:upperbound}).
\end{proof}

\remark{In particular, for $r= \left\lceil \left(\frac{4c_d}{\alpha}\right)^d \right\rceil$, there exist at least $\left( \frac{n}{2r} \right)^{r}$ $r$-sized subsets, each of whose convex hull contains all integer
points of depth at least $\alpha \cdot n$.}

\medskip

\begin{proof}[Proof of Theorem~\ref{thm:depthcovering}.]
Given the point set $P$ in $\Re^d$ and a point $\q \in \Re^d$ of half-space depth 
$\alpha \cdot n$, apply
Lemma~\ref{lemma:depthpartitioning} with $P$ and $\alpha$ to get a partition consisting
of $r \geq d + 1$ sets, 
where $r = \left\lceil \left(\frac{4c_d}{\alpha}\right)^d \right\rceil$.
By discarding at most $\frac{n}{2}$ points of $P$, we can derive a partition on the remaining points of $P$, say
$\mathcal{P} = \left\{P_1, \ldots, P_r \right\}$, such that the $P_i$'s are equal-sized disjoint subsets of $P$,
i.e., $|P_i| = \frac{n}{2r}$ for all $i = 1, \ldots, r$. Furthermore, 
every transversal of $\mathcal{P}$ contains all points in $\Re^d$ of half-space depth at least $\alpha n$,
and thus $\q$.

For each transversal $Q$ of $\mathcal{P}$, 
the point $\q$ lies in the convex hull of $Q$, and by Carath\'eodory's theorem, there exists
a $(d+1)$-sized subset of $Q$ whose convex hull also contains $\q$. 
By the pigeonhole principle, there must exist $(d+1)$ sets of $\mathcal{P}$, say
the sets $P_1, \ldots, P_{d+1}$, such that at least
\begin{equation} 
\frac{ \left(\frac{n}{2r}\right)^r }{\binom{r}{d+1} \left(\frac{n}{2r}\right)^{r-(d+1)}} 
\geq \frac{ \left(\frac{n}{2r}\right)^{(d+1)}}{ \left(\frac{er}{d+1}\right)^{d+1}}
= \frac{1}{\left(\frac{er}{d+1}\right)^{d+1}} \cdot \prod_{i=1}^{d+1} |P_i|.
\label{eq:densityarg}
\end{equation}
distinct transversals of $\left\{P_1, \ldots, P_{d+1} \right\}$ contain $\q$. 

\begin{sloppypar}

The rest of the proof  follows the one of Pach~\cite{P98}. In brief, we view the $P_i$'s as parts of a $(d+1)$-partite hypergraph with vertices corresponding to points in  $P$ and a hyperedge corresponding to each transversal of $\mathcal{P}$ containing $\q$. As there are $\Omega\left(n^{d+1}\right)$ such transversals by inequality~(\ref{eq:densityarg}), we apply a weak form of the hypergraph version of Szemer\'edi's regularity
lemma (see \cite{M92} Theorem 9.4.1) to derive
the existence of constant-fraction sized subsets  $P'_1 \subseteq P_1, \ldots, P'_{d+1} \subseteq P_{d+1}$ such that the following is true, for some constant $c'_d$:

\mystate{for \emph{any}   $P''_1 \subseteq P'_1, \ldots, P''_{d+1} \subseteq P'_{d+1}$,  with $|P''_i| \geq c'_d \cdot |P'_i|$ for $i = 1, \ldots, d+1$, we have the property
that there exists at least one transversal of $\left\{ P''_1, \ldots, P''_{d+1} \right\}$ whose convex hull contains $\q$.}

Then the same-type lemma (\cite{Barany:1998gi} Theorem 2) applied
to $\left\{P'_1, \ldots, P'_{d+1}, \left\{\q\right\} \right\}$ 
gives constant-fraction sized subsets $X_1 \subseteq P'_1, \ldots, X_{d+1} \subseteq P'_{d+1}$ such that each transversal of
$\left\{X_1, \ldots, X_{d+1} \right\}$ has the same order type with respect to $\q$. 
\end{sloppypar}

We can set up the parameters for the same-type lemma and the weak regularity lemma such
that $|X_i| \geq c'_d \cdot |P'_i|$, for all $i = 1, \ldots, d+1$.
Then the weak regularity lemma implies that there exists 
at least one transversal of $\left\{X_1, \ldots, X_{d+1}\right\}$
that contains $\q$. However, as each transversal of $\left\{X_1, \ldots, X_{d+1}\right\}$ has the same order type, 
it must be that \emph{each} transversal
of $\left\{X_1, \ldots, X_{d+1} \right\}$   contains $\q$. These are the required subsets.

The size of each $X_i$ is a constant-fraction of $n$, say $|X_i| \geq c_{d, \alpha} \cdot n$,
where the constant $c_{d, \alpha}$ depends on the constants in inequality~(\ref{eq:densityarg}),  
in the weak regularity lemma and in the same-type lemma. All of these depend only on $\alpha$ and $d$.
\end{proof}

\section{Acknowledgments} 
This work was partially supported by NSF grant DMS-1440140, while the first and third authors were in residence at the 
Mathematical Sciences Research Institute in Berkeley,  California, during the Fall 2017 semester. The first and second 
authors were also partially supported by NSF grants DMS-1522158 and DMS-1818169. The fourth author was supported by 
ANR grant SAGA (ANR-14-CE25-0016). Last, but not least, we are grateful for the information and comments we received 
from  N. Amenta, G. Averkov, A. Basu, J.P. Doignon, P. Sober\'on, D. Oliveros, and S. Onn. In particular we thank Pablo Sober\'on 
for noticing that the same methods used in an earlier version could be extended to yield Theorem~\ref{Stverberg}.
\bibliographystyle{amsalpha}

\bibliography{references__1_}

\end{document}